\documentclass[10pt]{article}
\pdfoutput=1
\usepackage{amsmath, amsthm,amssymb,enumerate,hyperref, mathtools, color, tikz, tikz-cd, calc, float, mathrsfs, url, stmaryrd}

\usepackage{caption}
\usepackage{subcaption}

\usetikzlibrary{patterns}

\usepackage[style=numeric, firstinits=true ,hyperref, backend=bibtex, doi=false, url=false, isbn=false, date=year]{biblatex}

\renewbibmacro{in:}{%
\ifentrytype{article}{}{\printtext{\bibstring{in}\intitlepunct}}}

\DeclareFieldFormat[article, misc, incollection, inproceedings, book]{title}{#1}

\DeclareFieldFormat[article, misc]{volume}{\mkbibbold{#1}}

\DeclareFieldFormat[book]{series}{\textit{#1}}

\DeclareFieldFormat[article, inbook, incollection, inproceedings, misc, thesis, unpublished]{pages}{#1}

\DeclareFieldFormat[article, inbook, incollection, inproceedings, misc, thesis, unpublished]{number}{}

\DeclareFieldFormat[inbook, incollection, book]{date}{\mkbibparens{#1}}

\renewbibmacro*{publisher+location+date}{%
  \printlist{publisher}%
  \iflistundef{location}
    {\setunit*{\space}}
    {\setunit*{\addcomma\space}}%
  \printlist{location}%
  \setunit*{\space}%
  \usebibmacro{date}%
  \newunit}

\renewbibmacro*{volume+number+eid}{%
  \printfield{volume}%
  \printfield{number}
  \setunit{\addcomma\space}%
  \printfield{eid}}

\renewcommand{\epsilon}{\varepsilon}
\newcommand{\e}{\epsilon}
\newcommand{\id}{\operatorname{id}}
\renewcommand{\d}{\partial}
\renewcommand{\to}{\longrightarrow}
\newcommand{\inv}{{}^{-1}} 
\newcommand{\ldotss}{, \ldots, } 

\newcommand{\R}{\mathbb{R}} 
\renewcommand{\Rn}{\R^n}

\renewcommand{\RN}{\R^N}

\newcommand{\bbZ}{\mathbb{Z}} 
\newcommand{\bbN}{\mathbb{N}} 

\newcommand{\calH}{\mathscr{H}} 
\newcommand{\calL}{\mathcal{L}} 

\newcommand{\calD}{\mathscr{D}} 

\newcommand{\calI}{\mathcal{I}} 

\newcommand{\textL}{\operatorname{L}} 
\renewcommand{\l}{l^\infty(\bbN)}

\newcommand{\Brace}[1]{\Big\lbrace #1\Big\rbrace} 
\renewcommand{\brace}[1]{\lbrace #1\rbrace} 
\newcommand{\sthat}{\:\:\colon\:\:} 
\newcommand{\q}[1]{\left[ #1\right]} 

\newcommand{\B}[2]{B_{#1}({#2})} 
\renewcommand{\r}[1]{\big\vert_{#1}} 
\renewcommand{\bar}[1]{\overline{#1}}
\newcommand{\abs}[1]{\big\vert #1 \big\vert} 
\newcommand{\norm}[1]{\Vert #1\Vert} 

\renewcommand{\i}{\q{0,1}}

\newcommand{\Lip}{\operatorname{Lip}}  

\newcommand{\dimN}{\operatorname{dim_{N}}}

\newcommand*\diff{\mathop{}\!\mathrm{d}}
\newcommand{\dx}{\diff x} 

\newcommand{\wedges}{\wedge\ldots\wedge} 

\newcommand{\loc}{\operatorname{loc}} 
\renewcommand{\c}{\operatorname{c}} 
\newcommand{\M}{\operatorname{\mathbf{M}}}  
\newcommand{\N}{\operatorname{\mathbf{N}}} 

\newcommand{\I}[1]{\operatorname{\mathbf{I}}_{#1, \c}} 
\newcommand{\Z}[1]{\operatorname{\mathbf{Z}}_{#1, \c}} 
\renewcommand{\P}[1]{\operatorname{\mathbf{P}}_{#1,\c}} 
\renewcommand{\L}[1]{\operatorname{\mathbf{L}}_{#1, \c}} 

\newcommand{\spt}{\operatorname{spt}}  
\newcommand{\push}[1]{#1_{\#}} 
\newcommand{\qq}[1]{\llbracket #1 \rrbracket} 
\newcommand{\ii}{\qq{0,1}}
\newcommand{\rr}{\mathbin{\vrule height 1.6ex depth 0pt width 0.1ex
						\vrule height 0.1ex depth 0pt width 1ex}} 

\newcommand{\V}{\operatorname{\mathbf{V}}} 

\newcommand{\IFF}[1]{\operatorname{\mathbf{I}}_{#1, \c}^{\operatorname{FF}}}

\newcommand{\Addresses}{{
  \bigskip
  \footnotesize

  \textsc{Department of Mathematics, ETH Z\"urich,
    R\"amistrasse 101, 8092 Z\"urich}\par\nopagebreak
  \textit{E-mail address}: \texttt{tommaso.goldhirsch@math.ethz.ch}
}}

\newtheorem{theorem}{Theorem}[section]
\newtheorem{corollary}[theorem]{Corollary}
\newtheorem{proposition}[theorem]{Proposition}
\newtheorem{lemma}[theorem]{Lemma}

\newtheorem*{theorem*}{Theorem}

\theoremstyle{definition}
\newtheorem{definition}[theorem]{Definition}

\newtheorem{example}[theorem]{Example}

\newtheorem*{exercise*}{Exercise}

\theoremstyle{remark}

\newtheorem*{remark*}{Remark}
\newtheorem*{claim*}{Claim}

\title{Lipschitz Chain Approximation\\ of Metric Integral Currents}

\date{May 7, 2021}
\author{Tommaso Goldhirsch\thanks{
Research supported by the Swiss National Science Foundation.
}}

\begin{document}
\maketitle
\begin{abstract}
Every integral current in a locally compact metric space $X$ can be approximated 
by a Lipschitz chain with respect to the normal mass,
provided that Lipschitz maps into $X$ can be extended slightly.
\end{abstract}

\tableofcontents

\newpage
\section{Introduction}
In \cite{Federer1960}, after proving 
their celebrated deformation theorem,
Federer and Fleming show 
that every integral current admits an approximation
by Lipschitz chains. More precisely,
for each current $T\in \I{n}(\RN)$ and $\e>0$
there is a Lipschitz chain $P\in \L{n}(\RN)$ such that $\N(T-P)<\e$, see Theorem 5.8 in \cite{Federer1960}. 

In this paper we prove an analogue of this result
for a locally compact metric space $X$ with
the property that every Lipschitz map into $X$ can be extended to a neighborhood of its domain.
In fact, we need this property to hold only locally and for Lipschitz maps with compact domains.
We will work in the context of metric \textit{integer rectifiable currents} $\calI_{n}(X)$
and metric \textit{integral currents} $\I{n}(X)$,
see \cite{Ambrosio2000} and \cite{Lang2011}.
All relevant concepts and results will be discussed in Section \ref{section-preliminaries}.
\vspace{0.2cm}

The abelian group $\calI_{n}(X)$ of integer rectifiable currents is equipped
with the \textit{mass} norm $\M$ and 
its subgroup $\I{n}(X)$ of integral currents consists of all
currents $T\in \calI_{n}(X)$ with compact support
and whose \textit{boundary} $\d T$ is also integer rectifiable, that is, an element of $\calI_{n-1}(X)$.
The resulting
chain complex $\I{\ast}(X)$ of integral currents is endowed with a norm $\M$ in each degree and boundary
maps $\d\colon \I{k}(X)\to \I{k-1}(X)$,
so that each current $T\in\I{n}(X)$ has 
\textit{mass} $\M(T)$ and \textit{normal mass} $\N(T)=\M(T)+\M(\d T)$.
A current $Z\in \I{n}(X)$ with $\d Z=0$ is called a \textit{cycle} and we 
denote by $\Z{n}(X)$ the corresponding subgroup of cycles.
An element
$V\in\I{n+1}(X)$ is called a \textit{filling} of
$S\in \Z{n}(X)$ if $\d V=S$.

Every singular Lipschitz chain in $X$ with integer coefficients induces
a current $P\in \L{n}(X)$, called \textit{Lipschitz chain}, which is an element of $\I{n}(X)$.

If $X=\RN$ there is a canonical isomorphism between metric currents in $\I{n}(\RN)$
and classical integral currents from \cite{Federer1960} (see Section 5 in \cite{Lang2011}).

\vspace{0.2cm}
We begin with a weaker approximation result for the $\M$-norm.
Every compactly supported function $u\in\textL^1(\Rn)$ induces an integer 
rectifiable current $\qq u\in\calI_{n,\c}(\Rn)$ with 
$\M(\qq u)=\norm{u}_{\textL^1}$.
If $F\colon \Rn\to X$ is $L$-Lipschitz, then 
there is a push-forward
$\push F\qq u\in \calI_{n,\c}(X)$, and
$\M(\push F\qq{u})\leq L^n\M(\qq u)$.

 Every current $T\in \calI_{n}(X)$ can be written
 as $T=\sum_{i}T_i$, where the sum converges with respect to the mass norm
 and each $T_i\in \calI_{n,\c}(X)$ is of the form $T_i=\push{(F_i)}\qq{u_i}$
for some integer valued $u_i\in \textL^1(\Rn)$, $F_i\colon K_i\to F_i(K_i)\subset X$ bi-Lipschitz 
and $K_i\subset \Rn$ compact containing the support of $u_i$.
By a purely measure-theoretic argument 
it is possible to approximate each $u_i$ in the $\textL^1$-norm with
a finite sum of characteristic functions corresponding to Borel subsets $B_j$ contained in $\spt(u_i)$.
In turn, these Borel subsets can be approximated by cubes.
This produces an approximation of $\qq{u_i}$ by Lipschitz chains with respect to the $\M$-norm in $\Rn$.
However, these cubes may leak slightly outside of $B_j$ and in particular outside of $\spt(u_i)$,
so that their image in $X$ is not defined.
For this reason, we assume that all such maps can be extended to a (small) neighborhood
of their domain.

\newpage
\begin{definition}
A metric space $X$ has \textit{property $L$} if the following holds.
For every metric space $Y$, every compact subset $K\subset Y$, and every $1$-Lipschitz map 
$g\colon K\to X$, there exist $\e=\e(g)>0$ and $L=L(g)\geq 1$ such that $g$ admits an
$L$-Lipschitz extension $\bar{g}\colon K_\e\to X$,
where $K_\e$ denotes the open $\e$-neighborhood of $K$.
A metric space $X$ has \textit{local property $L$} if every point in $X$ has a neighborhood with property $L$.
\end{definition}

If a metric space $X$ has property $L$, then every open subset has property $L$.
In particular property $L$ implies local property $L$.

Both conditions imply that $X$ is semi-locally quasi-convex (see Lemma \ref{lemma-quasi-convex}) and
$X$ has property $L$ for example if $X$, or each compact subset of $X$,
is an absolute Lipschitz (neighborhood) retract, but the converse need not be true.
A more detailed discussion can be found in Section \ref{section-Lipschitz-extensions}.

The argument above directly implies that if $T\in\calI_n(X)$
has support contained in an open subset of $X$ with property $L$, then it admits an approximation
by Lipschitz chains with respect to the $\M$-norm.
In general, however, if we assume that $X$ has local property $L$, 
neither $T$ nor the $T_i$'s
have support inside such an open subset of $X$ and this argument needs a small refinement:
each $T_i$ can be decomposed as a finite sum 
$T_i=T_i^1+\cdots +T_i^{N_i}$ with each $T_i^j$ having support 
inside an open subset of $X$ having property $L$ (see Lemma \ref{lemma-decomposition}).
In this situation we can apply the partial result just outlined for each $T_i^j$ and obtain the following.

\begin{proposition}[$\M$-Approximation]\label{prop_M_approximation}
Let $n\geq 1$, let $X$ be a locally compact metric space with local property $L$, and let 
$T\in \calI_n(X)$.
Then for every 
$\e>0$ and open subset $U\subset X$ with $\spt(T)\subset U$
there is
$P\in\L{n}(X)$ with $\spt(P)\subset U$ and ${\M(T-P)<\e}$.
\end{proposition}

This is false without (local) property $L$:
consider the fat Cantor set $X\subset \R$ of Lebesgue measure $\tfrac{1}{2}$,
which is compact and has empty interior.
Then $\qq{X}\in\calI_{1,\c}(X)$ is not trivial, but it cannot be approximated 
by $1$-dimensional Lipschitz chains in $X$.

Observe that while $\M(T-P)$ is arbitrarily small, we do not
control $\M(\d T-\d P)$.
The main part of
this paper deals with upgrading this mass approximation to a normal mass approximation.
We shall follow Federer and Fleming's approach in \cite{Federer1960}:
as a consequence of the deformation theorem, they first prove that
if a current $T\in \I{n}(\RN)$ has
boundary $\d T\in \L{n-1}(\RN)$, it can be deformed into a Lipschitz chain 
(compare with Lemma 5.7 in \cite{Federer1960}).

\begin{lemma}\label{lemma_FF_Lemma5.7_0}
Let $N>n\geq 1$ and let $T\in \I{n}(\RN)$ with $\d T\in \L{n-1}(\RN)$.
Then for every $\eta>0$ there is $R\in\I{n+1}(\RN)$ with $T-\d R\in \L{n}(\RN)$, $\N(R)\leq \eta$ and
$\spt(R)\subset \spt(T)_\eta$.
\end{lemma}

By smartly applying Lemma \ref{lemma_FF_Lemma5.7_0} twice, once in dimension $n-1$ and once in dimension $n$,
they finally prove the $\N$-approximation theorem for $\RN$.
This argument does not translate one to one in our setting,
as there is no analogue of the deformation theorem for general metric spaces.
The deformation theorem is a powerful result in the classical theory of currents.
It provides a way of deforming a current into the $n$-skeleton of the cube decomposition of $\RN$,
while keeping uniform bounds on the masses of the currents involved 
(see Theorem 5.5 in \cite{Federer1960}, also Theorem 4.2.9 in \cite{Federer1996}
and Theorem 1, Section 2.6 in \cite{Giaquinta1998}).

We solve this issue taking inspiration from De Pauw's strategy in \cite{DePauw2014}.
We will embed a compact neighborhood $K$ of $\spt(T)$ into $\l$ and exploit the metric approximation property of $\l$
to project it onto a finite dimensional vector subspace, 
in which we can apply Lemma \ref{lemma_FF_Lemma5.7_0}.

To be more precise, suppose first that $K$ is contained in 
an open subset of $X$ having property $L$ and
let $\iota \colon K\to \l$ be an isometric embedding with image $K'\coloneqq \iota(K)$.
By the metric approximation property of $\l$
there is a finite dimensional subspace $V\subset\l$
arbitrarily close to $K'$,
and a $1$-Lipschitz projection $\pi\colon K'\to V$.
In particular, we can choose $V$ close enough such that the extension of $\iota\inv$, say $g$, provided by property $L$ is 
defined on $\pi(K')$.

Ideally, we could consider the projection $T''=\push{(\pi\circ\iota)}T$ of $T$ in $\pi(K')$, apply Federer and Fleming's
$\N$-approximation theorem in $V$ and then map back to $X$ using the extension $g$ of $\iota\inv$.
The issue is that by doing so we do not have 
control over the difference between the original current $T$ and $\push g T''$.

This issue can be overcome 
by using Lemma \ref{lemma_FF_Lemma5.7_0} instead of directly applying the approximation result
for the $\N$-norm,
together with the construction of a "homotopy filling" which
allows us to control the error produced by $\push g$. 
(See Section \ref{section-homotopies}.)

Similarly to the approximation with respect to the $\M$-norm, this argument 
proves that every current $T\in\I{n}(X)$ 
with support contained in an open subset of $X$ with property $L$
can be approximated by Lipschitz chains with respect to the $\N$-norm.
As before, assuming that $X$ has local property $L$, 
we can apply Lemma \ref{lemma-decomposition}
to
write an arbitrary current $T\in\I{n}(X)$ as $T=T_1+\cdots +T_k$ where each $T_i\in\I{n}(X)$ has support in an open subset of 
$X$ having property $L$,
then using the partial result just described
we can prove the following analogue of Theorem 5.8 in \cite{Federer1960}.

\begin{theorem}[$\N$-Approximation]\label{theorem-N-approx}
Let $n\geq 1$, let $X$ be a locally compact metric space with local property $L$, and 
let $T\in\I{n}(X)$. Then for every $\e>0$ there is $P\in\L{n}(X)$ with $\N(T-P)<\e$ and $\spt(P)\subset \spt(T)_\e$.
\end{theorem}

In the case $n=1$, that is $T\in \I{1}(X)$,
 we can actually reach a stronger conclusion where the approximating chain $P\in \L{1}(X)$ 
 not only satisfies ${\N(T-P)<\e}$ and $\spt(P)\subset \spt(T)_\e$, 
 but also $\d P=\d T$
 (see Corollary \ref{cor-N-approx-n=1}).
 
\paragraph*{Acknowledgements}
I am indebted to Urs Lang for suggesting the problem
and providing key insights while writing this paper.
I also wish to thank Lisa Ricci for many helpful discussions.

\newpage

\section{Preliminaries}\label{section-preliminaries}
Let $X=(X,d)$ be a metric space. We write 
$\B{x}{r}\coloneqq \brace{y\in X\sthat d(x,y)\leq r}$ for the closed ball
of radius $r\geq 0$ and center $x\in X$.

Given a subset $A\subset X$ and $\e>0$ we denote by 
$A_\e\coloneqq \brace{x\in X\sthat d(A,x)<\e}$ the open $\e$-neighborhood of $A$ in $X$,
where $d(A,x)$ is the infimum over all $d(a,x)$ with $a\in A$.

A map $f\colon X\to Y$ into another metric space $Y=(Y,d)$ is $L$-\textit{Lipschitz},
for some constant $L\geq 0$, if $d(f(x), f(x'))\leq L d(x,x')$ for all $x,x'\in X$.
The \textit{Lipschitz constant} $\Lip(f)$ of $f$ is the infimum over all such $L$.
A map $f$ is \textit{Lipschitz} if it is $L$-Lipschitz for some $L$, it is \textit{locally Lipschitz} if every point in $X$ has a neighborhood 
on which $f$ is Lipschitz.
We denote by $\Lip_{\loc}(X)$ and $\Lip_{\c}(X)$ the spaces of functions $X\to\R$ which are locally Lipschitz 
or Lipschitz with compact support, respectively.
A map $f$ is a \textit{bi-Lipschitz} embedding if it is injective and both $f$ and $f\inv$ are Lipschitz.

\subsection{Metric Currents}
Metric currents of finite mass were introduced by Ambrosio and Kirchheim in \cite{Ambrosio2000}.
Here we will work with a variant of this theory for locally compact metric spaces, as described by Lang in \cite{Lang2011}.
In this section we provide some background on this theory and refer the reader to
\cite{Lang2011} for more details.
We will assume throughout that the underlying metric space $X$ is locally compact.

For every integer $n\geq 0$ let 
$\calD^n(X)\coloneqq \Lip_{\c}(X)\times \left[\Lip_{\loc}(X)\right]^n$, that is,
the set of $(n+1)$-tuples $(f,\pi_1\ldotss\pi_{n})$ of real valued functions on $X$ such that $f$ is Lipschitz with compact support and 
$\pi_1\ldotss \pi_n$ are locally Lipschitz.
We will use $(f,\pi)$ as a shorthand for $(f,\pi_1\ldotss\pi_n)$.
The idea is that $(f,\pi_1\ldotss\pi_n)\in \calD^n(X)$ represents
the compactly supported differential $n$-form 
$fd\pi_1\wedges d\pi_n$ if $X$ is (an open subset of) $\R^N$ and the 
functions
$f,\pi_1\ldotss\pi_n$ are smooth;
and roughly speaking, a current (with some additional properties defined below)
is a map $\calD^n(X)\to \R$ representing integration 
on a submanifold of $\R^N$.

\begin{definition}
An $n$-dimensional \textit{current} $T$ on $X$ is a function 
${T\colon \calD^n(X)\to\R}$ satisfying the following three properties:
\begin{itemize}
\item[(1)] (multilinearity) $T$ is $(n+1)$-linear;
\item[(2)] (continuity) if ${f^k\rightarrow f}$, ${\pi_i^k\rightarrow \pi_i}$
pointwise on $X$, $\sup_k\Lip(f^k\r{K})<\infty$, $\sup_k\Lip(\pi_i^k\r{K})<\infty$ for every compact set 
$K\subset X$ (for each $i=1\ldotss n$) and $\bigcup_k\spt(f^k)\subset K$ for some compact set $K\subset X$, then
\[
T(f^k,\pi_{1}^k\ldotss\pi_n^k)\to T(f,\pi_1\ldotss\pi_n);
\]
\item[(3)] (locality) $T(f,\pi_1\ldotss\pi_n)=0$ whenever one of the functions 
$\pi_1\ldotss\pi_n$ is constant on a neighborhood of $\spt(f)$.
\end{itemize}
\end{definition}
The vector space of all $n$-dimensional currents on $X$ is denoted by $\calD_n(X)$.
Every function $u\in\textL^1_{\loc}(\Rn)$ induces a current $\qq u\in\calD_n(\Rn)$ defined by
\[
\qq u(f,\pi_1\ldotss\pi_n)\coloneqq
\int u f\det\Big(\frac{\d \pi_i}{\d x^j}\Big)_{i,j=1}^n\dx=\int uf\det(D\pi)\dx
\]
for all $(f,\pi_1\ldotss\pi_n)\in\calD^n(\Rn)$, where the partial derivatives
in the Jacobian $D\pi$ of $\pi=(\pi_1\ldotss\pi_n)$ exist almost everywhere according to
Rademacher's theorem (see Proposition 2.6 in \cite{Lang2011}).
This corresponds to the integration of the differential form 
$f d\pi_1\wedges\pi_n$ over $\Rn$, multiplied by $u$.
If $W\subset \Rn$ is a Borel set and $\chi_W$ is its characteristic function,
we set $\qq W\coloneqq \qq{\chi_W}$.

\subsection{Support, Push-forward, and Boundary}
Let $T\in\calD_n(X)$ be an $n$-dimensional current.
The \textit{support} $\spt(T)$ of $T$ is the smallest closed subset of $X$
such that the value $T(f,\pi_1\ldotss\pi_n)$ depends only on the restrictions of
$f,\pi_1\ldotss\pi_n$ to it.

For a proper Lipschitz map $F\colon \spt(T)\to Y$ 
into another locally compact metric space
$Y$, 
the \textit{push-forward} $\push{F}T\in\calD_n(Y)$ is defined by
\[
\push F (f,\pi_1\ldotss\pi_n)\coloneqq T(f\circ F,\pi_1\circ F\ldotss \pi_n\circ F)
\]
for all $(f,\pi)\in\calD^n(Y)$.
It holds that $\spt(\push F T)\subset F(\spt(T))$.

For $n\geq 1$, the \textit{boundary} 
$\d T\in\calD_{n-1}(X)$ of $T$ is defined by
\[
(\d T)(f,\pi_1\ldotss \pi_{n-1})=
T(\sigma,f,\pi_1\ldotss\pi_{n-1})
\]
for $(f,\pi_1\ldotss\pi_{n-1})\in\calD^{n-1}(X)$, where $\sigma$ is any
compactly supported
Lipschitz function, which is identically $1$ on $\spt(f)\cap\spt(T)$.
It holds that $\d\circ \d=0$, $\spt(\d T)\subset \spt(T)$, and 
$\push  F(\d T)=\d(\push F T)$ for $F$ as above.
(For more details, see Section 3 in \cite{Lang2011}.)

\subsection{Mass}
Let $T\in\calD_n(X)$ be an $n$-dimensional current.
For an open set $U\subset X$, the \textit{mass}
$\norm{T}(U)\in\q{0,\infty}$ of $T$ in $U$ is defined as the supremum of
$\sum_{i=1}^k T(f^i,\pi_{1}^i\ldotss\pi_n^i)$
over all finite families $(f^i,\pi_{1}^i\ldotss\pi_n^i)_{i=1}^k\subset \calD^n(X)$
such that the restrictions of $\pi_{1}^i\ldotss\pi_n^i$ to $\spt(f^i)$ are 
$1$-Lipschitz for all $i$, $\bigcup_{i=1}^k\spt(f^i)\subset U$ and 
${\sum_{i=1}^k\abs{f^i}\leq 1}$.

This defines a regular Borel measure $\norm{T}$ on $X$.
The total mass $\norm{T}(X)$ of $T$ is denoted $\M(T)$ and is called the \textit{mass} of $T$.
If $S\in\calD_n(X)$ is another current, then $\norm{T+S}\leq\norm{T}+\norm{S}$, and in particular
\[
\M(T+S)\leq \M(T)+\M(S).
\]
For every Borel set $B\subset X$, it is possible to define the \textit{restriction} 
$T\rr{B}\in\calD_n(X)$ of $T$ to $B$;
the measure $\norm{T\rr{B}}$ coincides with the restriction $\norm{T}\rr{B}$
of the measure $\norm{T}$.

If $\M(T)<\infty$ and $F\colon \spt(T)\to Y$ is a proper $L$-Lipschitz map into a locally compact metric $Y$,
 and $B'\subset Y$ is a Borel set, then 
$(\push{F}T)\rr{B'}=\push{F}(T\rr{F\inv(B')})$
and $\norm{\push T}(B')\leq L^n\norm{T}(F\inv(B'))$, in particular,
\[
\M(\push F T)\leq L^n\M(T).
\]

The normal mass of $T$ in $U\subset X$ is defined as 
$\norm{T}(U)+\norm{\d T}(U)$ and the \textit{normal mass} of $T$ as
$\N(T)\coloneqq \M(T)+\M(\d T)$.
We say that $T$ 
is \textit{normal} if $\N(T)<\infty$.

The real vector space of all $T\in\calD_n(X)$ with finite mass is
denoted by $\M_{n}(X)$, and the subspace consisting of currents with compact support is denoted by $\M_{n,\c}(X)$;
they are both normed vector spaces when endowed with the mass norm $\M$. 
Similarly, $\N_{n}(X)$ and $\N_{n,\c}(X)$ 
denote the vector spaces of normal currents and normal currents with
compact support, respectively.

If $u\in\textL^1(\Rn)$, then $\M(\qq{u})=\norm{u}_{L^1}$, in particular,
if $K\subset \Rn$ is a Borel set, then $\M(\qq{K})=\calL^n(K)$ and $\d\qq{K}$ has finite mass 
$\M(\d\qq{K})$
whenever 
$\chi_K$ has finite variation $\V(\chi_K)$. 
(For more details, see Sections 2, 4 and 7 of \cite{Lang2011}.)

\subsection{Integral Currents}
A subset $E\subset X$ is 
\textit{countably $n$-rectifiable}
if there are countably many Lipschitz maps
$F_i\colon A_i\to X$, $A_i\subset\Rn$, such that $E\subset \bigcup_i F_i(A_i)$.
The set $E\subset X$ is \textit{countably $\calH^n$-rectifiable} if there is
a countably $n$-rectifiable set $E'\subset X$ such that
$\calH^n(E\setminus E')=0$, where $\calH^n$ is the 
$n$-dimensional Hausdorff measure on $X$.

A current $T\in\M_{n}(X)$ is called \textit{integer rectifiable} if
$\norm{T}$ is
concentrated on some $\calH^n$-rectifiable Borel set $E\subset X$
and the following integer multiplicity condition holds:
for every Borel set $B\subset X$ with compact closure and 
for every Lipschitz map $\pi\colon X\to \Rn$
the push-forward $\push\pi(T\rr{B})\in\calD_n(\Rn)$
is of the form $\qq{u}$, for some integer valued 
$u=u_{B,\pi}\in \textL^1(\Rn)$.
The abelian group of \textit{integer rectifiable} $n$-currents in $X$ is denoted by
$\calI_{n}(X)$;
it is closed under
push-forwards and restrictions to Borel sets.
We write $\calI_{n,\c}(X)$ for the subgroup of 
integer rectifiable currents
with compact support.

A current $T\in\calI_{n,\c}(X)$ is an \textit{integral current} with compact support, or simply an integral current,
if whenever $n\geq 1$, its boundary $\d T$ is integer rectifiable as well.
We denote the corresponding abelian groups by $\I{n}(X)$, and observe that they form a chain complex.
It is worth mentioning that by Theorem 8.7 (boundary rectifiability) in \cite{Lang2011}, 
$T\in\calI_{n,\c}(X)$ is integral if 
$\M(\d T)<\infty$ (that is, $T$ is normal),
equivalently,
${\I{n}(X)=\calI_{n,\c}(X)\cap \N_{n,\c}(X)}$.

If $K\subset \R^n$ is a bounded Borel set, then $\qq{K}$
is an element of $\calI_{n,\c}(\Rn)$, and it is in $\mathbf{I}_{n,\c}(\Rn)$ 
whenever $\chi_K$ has finite variation $\V(\chi_K)$.

An integral current $T$ is a \textit{cycle} whenever $\d T=0$ and
we denote by 
$\Z{n}(X)\subset \I{n}(X)$
the subgroup of integral cycles.
An element of $\I{0}(X)$ is an integer linear
combination of currents of the form $\qq x$, where $x\in X$ and
$\qq{x}(f)=f(x)$ for all compactly supported Lipschitz functions $f\in\calD^0(X)$.
In this case $\Z{0}(X)\subset \I{0}(X)$ denotes the subgroup of
integer linear combinations whose coefficients sum to zero.
Note that $\d\colon \I{n}(X)\to\Z{n-1}(X)$ for all $n\geq 1$, 
and if $F\colon X\to Y$ is a proper Lipschitz map into a locally compact metric space $Y$,
then the push-forward $\push{F}$ maps $\I{n}(X)$ to $\I{n}(Y)$
and $\Z{n}(X)$ to $\Z{n}(Y)$.
Given $Z\in\Z{n}(X)$ we call a current $V\in\I{n+1}(X)$ a \textit{filling}
of $Z$ if $\d V=Z$. (For more details, see Section 8 of \cite{Lang2011}.)
\vspace{0.1cm}

In general the restriction $T\rr{B}$ of an integral current $T$ to an arbitrary Borel subset $B\subset X$ 
is not integral but only integer rectifiable.
However, as a special case of a more general construction,
one can show that for $T\in\I{n}(X)$ and $x\in X$,
the restriction $T\rr{\B{x}{r}}$ is in $\I{n}(X)$ for almost every $r\geq 0$
(see Section 6 and Theorem 8.5 in \cite{Lang2011}, and Section 2.6 in \cite{KleinerLang2018}).

\subsection{Polyhedral and Lipschitz Chains}
An $n$-dimensional polyhedron $K$ in $\Rn$, such as a (hyper-)cube or an $n$-simplex,
 is the convex hull of finitely many
(non coplanar) points in $\Rn$.
As remarked before $\qq{K}$ is in $\calI_{n,c}(\R^n)$ and one can show that $\M(\d\qq{K})=\V(\chi_K)=\calH^{n-1}(\d K)<\infty$,
so that in fact $\qq{K}$ is in $\I{n}(\R^n)$.

A \textit{polyhedral $n$-chain} in $\Rn$ is a finite sum of the form
$P=\sum_{i=1}^l a_i\qq{D_i}$, where $a_i\in\bbZ$,
and $D_i\subset\Rn$ are $n$-dimensional polyhedra.

A \textit{Lipschitz $n$-chain} in $X$ is a finite 
sum of the form
\[
L=\sum_{i=1}^l a_i \push{(\varphi_i)}\qq{D_i},
\]
where $a_i\in\bbZ$, $D_i\subset \Rn$ are $n$-dimensional polyhedra and $\varphi_i\colon D_i\to X$ are Lipschitz maps.

We denote by $\P{n}(\Rn)$ and $\L{n}(X)$ the abelian groups of polyhedral $n$-chains in $\Rn$
and Lipschitz $n$-chains in $X$.
Rearranging the polyhedra $D_i$ in the definition,
one can show that every Lipschitz $n$-chain $L\in\L{n}(X)$
can be written as the push-forward $L=\push{\varphi}P$
of some polyhedral $n$-chain $P\in\P{n}(X)$.

In $\Rn$ every polyhedral $n$-chain is a Lipschitz chain
and in general every Lipschitz chain is an integral current,
that is, $\P{n}(\Rn)\subset \L{n}(\Rn)$ and $\L{n}(X)\subset \I{n}(X)$.
It is possible to define polyhedral $m$-chains in $\Rn$ also for $m<n$ but for our purposes
it suffices to consider them as Lipschitz $m$-chains in $\Rn$.

There is a chain isomorphism $\I{\ast}(\Rn)\to\IFF{\ast}(\Rn)$
between (metric) integral currents in $\Rn$ and "classical" 
Federer-Fleming integral currents of \cite{Federer1960}
which is bi-Lipschitz with respect to the $\M$-norm with constants depending only on the dimensions,
and which restricts to an isomorphism between the respective subchains 
of polyhedral and Lipschitz chains (see Theorem 5.5 in \cite{Lang2011}).
In particular we can apply Lemma \ref{lemma_FF_Lemma5.7_0} to metric integral currents in $\I{n}(\Rn)$.

Finally, note that all $0$-dimensional integral currents
are by definition Lipschitz chains, that is, $\I{0}(X)=\L{0}(X)$, and therefore
an approximation theorem for the $\N$-norm is not necessary in dimension $0$.

\subsection{Homotopies}\label{section-homotopies}
We recall a useful technique to produce fillings of cycles.
Let $X, Y$ be locally compact metric spaces 
and $H\colon \i\times X\to Y$ a Lipschitz homotopy between the
Lipschitz maps $f,g\colon X\to Y$.
Given $T\in\Z{n}(X)$, $n\geq 1$, 
one can construct the product current $\ii\times T\in \I{n+1}(\i\times X)$ 
and take the push-forward with respect to $H$.
This produces the current $\push H(\ii\times T)\in \I{n+1}(Y)$ with support
in $H(\i\times \spt(T))\subset Y$
and boundary
\[
\d\push H\big(\ii\times T)=\push g T-\push f T,
\]
where  $\push f T$ and $\push g T$ are in $\I{n}(Y)$.

If $H(t,\cdot)\r{\spt(T)}\colon \spt(T)\to Y$ is $L$-Lipschitz for all $t\in \i$ and
$H(\cdot, x)\colon \i\to Y$ is a geodesic of length at most $D$ for all $x\in \spt(T)$, then
\[
\M\big(\push H\big(\ii\times T\big)\big)\leq (n+1)L^n D\M(T).
\]
For more details see Theorem 2.9 in \cite{Wenger2005} and Section 2.7 of \cite{KleinerLang2018}.

Lipschitz chains are closed under products and push-forwards so that
$\ii\times P$ is in $\L{n+1}(\i \times X)$ and $\push H(\ii\times P)$ is in $\L{n+1}(Y)$
whenever $P$ is an element of $\L{n}(X)$.

\vspace{0.2cm}
We describe a special case of this construction.
Let $Y$ denote a normed vector space and $K\subset Y$ a compact subset.
Let $\varphi, \psi\colon K\to Y$ be $L$-Lipschitz maps
with $\abs{\psi(x)-\varphi(x)}\leq D$ for all $x\in K$,
and consider the affine homotopy $H\colon\i\times K\to Y$ from $\varphi$ to $\psi$, that is,
$H(t,x)\coloneqq t\psi(x)+(1-t)\varphi(x)$.
Then, if $P$ is an element of $\L{n}(K)$ with $\d P=0$,
the push-forward
$\push H(\ii\times P)\in \L{n+1}(Y)$ has support contained in $H(\i \times K)$
and satisfies
\begin{align*}
\d \push H\big(\ii\times P\big)&=\push\psi P-\push\varphi P,\\
\M\big(\push H\big(\ii\times P\big)\big)&\leq (n+1)L^n D\M(P).
\end{align*}
We call $\push H(\ii \times P)$ the \textit{affine (homotopy) filling} of $\push \psi P-\push \varphi P$.

\subsection{Finite Dimensional Projections}\label{section-finite-dimensional-projections}
As mentioned in the introduction, in order to exploit the deformation theorem we project integral currents defined 
on a metric space $X$
into  a finite dimensional vector space.
This is done in two steps.

First, every metric space $X$ embeds isometrically into the Banach space $l^\infty(X)$
of bounded maps on $X$
via the map $x\mapsto d(x,\cdot)-d(x_0,\cdot)$, for any base point $x_0\in X$.
If $X$ is compact, or more generally separable, and $(x_i)_{i\in \bbN}\subset X$ is a countable dense subset,
then $x\mapsto (d(x_i,x)_i-d(x_i,x_0))_{i\in \bbN}$ is an isometric embedding into $\l$,
and the second term $d(x_i,x_0)$ is not necessary if $X$ is bounded.
This allows us to embed a compact neighborhood of the support of $T\in\I{n}(X)$ into $\l$.

Then, we find a finite dimensional subspace of $\l$ which is "close enough" to 
the image of the embedding.
Recall that a Banach space $V$ has the \textit{bounded approximation property} 
if there exists $\lambda\geq 1$ such that the following holds.
For every compact subset $K\subset V$ and $\e>0$ there is a finite dimensional vector subspace $V'\leq V$
and a $\lambda$-Lipschitz map $\pi\colon K\to V'$ satisfying
$\abs{\pi(x)-x}\leq \e$ for all $x\in K$.
We say that $V$ has the \textit{metric approximation property} in the case $\lambda=1$.
Conveniently, $\l$ has this property.

\begin{proposition}\label{prop_map_0}
$\l$ has the metric approximation property.
\end{proposition}
For a detailed proof of this fact we refer to Theorem A.6 in \cite{DePauw2014}.
In the next section we discuss the property needed to go back from 
the finite dimensional subspace of $\l$ to $X$.

\section{Lipschitz Extensions}\label{section-Lipschitz-extensions}
We compare the Lipschitz extension property mentioned in the introduction
with other similar properties found in the literature.
In this section we do not assume $X$ to be locally compact, unless otherwise specified.
We recall the definition.

\begin{definition}
A metric space $X$ has \textit{property $L$} if the following holds.
For every metric space $Y$, every compact subset $K\subset Y$, and every $1$-Lipschitz map 
$g\colon K\to X$, there exist $\e=\e(g)>0$ and $L=L(g)\geq 1$ such that $g$ admits an
$L$-Lipschitz extension $\bar{g}\colon K_\e\to X$.
We say that $X$ has \textit{local property $L$} if every point in $X$ has a neighborhood with property $L$.
\end{definition}

\begin{lemma}\label{lemma-quasi-convex}
Let $X$ be a locally compact metric space with property $L$.
Then $X$ is semi-locally quasi-convex, that is, for every point $o\in X$ there are 
constants $r=r(o)>0$ and $L=L(o)\geq 1$ such that any two points $x,y\in\B{o}{r}$ are joined by a curve 
of length $\leq Ld(x,y)$ contained in $\B{o}{2Lr}$.
\end{lemma}

Suppose that $X$ is a locally compact metric space with local property $L$. Then each point 
has an \textit{open} neighborhood $U$ which is locally compact and has property $L$
and hence
this lemma implies that $X$ is semi-locally quasi-convex.

\begin{proof}
Let $o\in X$ and take $\delta>0$ small enough such that $K\coloneqq \B{o}{\delta}$ is compact.
Consider the isometric embedding $\iota\colon K\to \l$ with image $K'\coloneqq \iota(K)$.
By assumption there exist $\e>0$, $L\geq 1$ and an $L$-Lipschitz extension $g\colon K'_{\e}\to X$
of $\iota\inv\colon K'\to X$.

Let $r\coloneqq \min\brace{\tfrac{\e}{2},\delta}$ and consider the possibly smaller ball $\B{o}{r}$.
For $x,y\in \B{o}{r}$ let $\gamma\colon\i\to \l$ be 
the straight segment $\gamma(t)\coloneqq \iota(x)+t(\iota(y)-\iota(x))$ from $\iota(x)$ to $\iota(y)$
of length $\abs{\iota(x)-\iota(y)}=d(x,y)\leq 2r$.
The image of $\gamma$ is within distance at most $r$ from $\brace{\iota(x),\iota(y)}\subset K'$ and 
hence contained in $K'_\e$.
Thus $g\circ\gamma\colon \i\to X$ is a curve from $x$ to $y$ of length at most $Ld(x,y)$
and contained in $\B{o}{r+Lr}\subset \B{o}{2Lr}$.
\end{proof}

\vspace{0.1cm}
We now explore stronger properties.
\begin{definition}
A metric space $X$ is an \textit{absolute $1$-Lipschitz retract} 
if whenever $\iota\colon X\to Y$ is an isometric embedding into a metric space $Y$,
there is a $1$-Lipschitz retraction $\phi\colon Y\to \iota(X)$.
A metric space $X$ is \textit{injective} if for every metric space $B$ and every subset $A\subset B$,
every $1$-Lipschitz map $f\colon A\to X$ admits a $1$-Lipschitz extension $\bar{f}\colon B\to X$.
\end{definition}

Basic examples of injective metric spaces are $\R$ and the Banach space $l^\infty(J)$ of bounded functions on a set $J$.
Exploiting the isometric embedding of $X$ into the injective space $l^\infty(X)$, one can prove the following
equivalence (see for example Proposition 2.2 in \cite{Lang2013}).

\begin{proposition}\label{prop-short-injective-iff-alr}
A metric space $X$ is injective if and only if it is an absolute $1$-Lipschitz retract.
\end{proposition}

Since every separable metric space, in particular every compact metric space, embeds isometrically 
into $\l$, the same argument used in this proof can be applied to show that property $L$
is equivalent to the following apparently weaker properties.

\begin{itemize}
\item[(1)] For every compact subset $K\subset\l$ and every $1$-Lipschitz map 
$g\colon K\to X$, there exist $\e=\e(g)>0$ and $L=L(g)\geq 1$ such that $g$ admits an
$L$-Lipschitz extension $\bar{g}\colon K_\e\to X$.
\item[(2)]
For every Banach space $V$, every compact subset $K\subset V$, and $1$-Lipschitz map
$g\colon K\to X$, there exist $\e=\e(g)>0$ and $L=L(g)\geq 1$ such that 
$g$ admits an $L$-Lipschitz extension $\bar{g}\colon K_\e\to X$.
\end{itemize}

All injective metric spaces, equivalently all absolute $1$-Lipschitz retracts, have property $L$.
In fact all metric spaces for which every compact subset is contained in an injective subspace,
have property $L$.
The converse is not necessarily true as the Lipschitz extension of property $L$
possibly has a worse Lipschitz constant and is defined only on a neighborhood of its original domain.

We say that a metric space is an \textit{absolute Lipschitz retract} (ALR) if the retraction 
$\phi\colon Y\to \iota(X)$ in the 
definition above is allowed to be Lipschitz (instead of $1$-Lipschitz).
Also in this case $X$ has property $L$ if it is an ALR (or all its compact subsets are contained in one).

Lang and Schlichenmaier \cite{Lang2005} provide an
instance in which $X$ is an ALR
and so has property $L$ (see Corollary 1.8 in \cite{Lang2005}).

\begin{theorem}
Suppose that $X$ is a
metric space with finite Nagata dimension  $\dimN(X)\leq n<\infty$.
Then $X$ is an ALR if and only if $X$ is complete and Lipschitz $n$-connected.
\end{theorem}

Recall that a family $\cal B$ of subsets of $X$ is \textit{$D$-bounded}, for some $D\geq 0$, 
if the diameter of all subsets in $\cal B$ is uniformly bounded by $D$.
For $s>0$, its  $s$-\textit{multiplicity} is the infimum over all integers $n\geq 0$
such that each subset of $X$ with diameter $\leq s$ meets at most 
$n$ members of $\cal B$.
The \textit{Nagata dimension} $\dimN(X)$ of $X$ is the infimum over all integers
$n$ with the following property:
there exists a constant $c>0$ such that for all $s>0$,
$X$ has a $cs$-bounded covering with $s$-multiplicity at most $n+1$.
For example, every doubling metric space has finite Nagata dimension.
(See Section 2 of \cite{Lang2005}.)

For an integer $m\geq 0$ denote by $S^m$ and $B^{m+1}$
the unit sphere and closed ball in $\R^{m+1}$, endowed with the induced metric.
A metric space $X$ is \textit{Lipschitz $n$-connected} for some integer $n\geq 0$ if
there is a constant $\gamma$ such that for every $m\in\brace{0,1\ldotss n}$,
every $\lambda$-Lipschitz map $f\colon S^m\to X$, $\lambda>0$,
admits a $\gamma\lambda$-Lipschitz extension $\bar{f}\colon B^{m+1}\to X$.
\vspace{0.2cm}

We now consider Lipschitz extensions 
which are possibly not defined on the whole ambient space of their domain.

\begin{definition}
A metric space $X$ is an \textit{absolute Lipschitz neighborhood retract} (ALNR) if whenever 
$\iota\colon X\to Y$ is an isometric embedding into a metric space $Y$,
there exist a neighborhood $W$ of $\iota(X)$ in $Y$ and a Lipschitz retraction 
$\pi\colon W\to \iota(X)$.

If the neighborhood $W$ is of the form $\iota(X)_\e$ for some $\e>0$ we say that 
$X$ is an \textit{absolute Lipschitz uniform neighborhood retract} (ALUNR).
\end{definition}

Using the injectivity of $l^\infty(X)$
one can show that if $X$ is an ALUNR, then there exists $\e_0$ such that 
the neighborhood $W$ is of the form $W=\iota(X)_{\e_0}$ for \textit{every}
metric space $Y$ and every isometric embedding
isometric embedding $\iota\colon X\to Y$.
Similarly to Proposition \ref{prop-short-injective-iff-alr} the 
following equivalence holds.

\begin{proposition}
A metric space $X$ is an ALNR if and only if 
for every metric space $B$ and every subset $A\subset B$,
every Lipschitz map $g\colon A\to X$ admits a Lipschitz extension
$\bar{g}\colon U\to X$ to a neighborhood $U$ of $A$ in $B$.
\end{proposition}

An analogous statement holds for ALUNR instead of ALNR.
Also in this case $X$ has property $L$ if it is an ALNR or even if
every compact set is contained in an ALNR subset of $X$.

In general, the opposite implication is not true because
property $L$ extends only Lipschitz maps whose domains
are compact sets.
Also, to prove that a subset $X'\subset X$ is an ALNR one needs to consider Lipschitz
maps into $X'$, but the images of their extensions provided by property $L$ 
might land outside $X'$. 

Another argument is that if $X'\subset X$ is an ALNR, then by definition
there exists an open neighborhood $W\subset X$ of $X'$ and a Lipschitz retraction
$\phi\colon W\to X'$.
This is a property that spaces with property $L$ generally do not have.

From the point of view of maps within $X$, property $L$ does not seem to be restrictive
in the sense that 
if $X$ has property $L$, then for each compact subset $K\subset X$ the inclusion extends to a 
Lipschitz map $W\to X$ for some neighborhood $W$ of $K$ in $X$. 
This map might as well be the identity $X\to X$, so all metric spaces satisfy this property.
\vspace{0.2cm}

In the context of continuous extensions of continuous maps,
ANRs are defined similarly to ALNRs where the maps involved are continuous and not necessarily Lipschitz.
If every $x\in X$ has an ANR neighborhood, then $X$ is an ANR (see for instance Theorem 3.2 in \cite{Hanner1951}).
A local to global result of this type 
does not hold for ALRs as Example \ref{example_2} below shows,
and we now indicate why it is not expected
for property $L$ since generally it is not possible to glue Lipschitz maps.

Let $K$ be a compact subset of a metric space $Y$ and let
$h\colon K\to X$ be a $1$-Lipschitz map with image $h(K)\subset U\cup V$, where both $U$ and $V$ are open subsets
of $X$ with property $L$.
Suppose that we can write $K=K_1\cup K_2$ where $K_1$ and $K_2$ are overlapping 
compact subsets with $h(K_1)\subset U$ and $h(K_2)\subset V$.
By assumption we can find Lipschitz extensions
$\bar{f}\colon (K_1)_\e\to X$, and ${\bar{g}\colon (K_2)_\e\to X}$
of $f\coloneqq h\r{K_1}$ and $g\coloneqq h\r{K_2}$, respectively.
The issues here are that the extensions might differ
on $(K)_\e\setminus (K_1\cap K_2)$,
and that 
 $\bar{f}, \bar{g}$ need not coincide with $h$ on $(K_1)_\e \cap K$ and on $(K_2)_\e\cap K$, respectively.
The following example shows that gluing two Lipschitz functions
agreeing on some subset of their respective domain in general does not yield a Lipschitz function
(compare with Section 2.5 of \cite{Cobza2019}).
\begin{example}
Let
\begin{align*}
C&\coloneqq\brace{(t,0)\sthat t\in \i}\subset \R^2,\\
D&\coloneqq\brace{(t,t^2)\sthat t\in \i}\subset \R^2,
\end{align*}
and define $f\colon C\cup D\to \R$ as $f(t,0)\coloneqq 0$, $f(t,t^2)\coloneqq t$. 
Then
$f\r{C}$ and $f\r{D}$ coincide in $(0,0)$ and are both $1$-Lipschitz,
but $f$ itself is not.
Indeed, if it were $L$-Lipschitz for some $L\geq 1$, then 
\[
t=d\big(f(t,t^2), f(t,0)\big)\leq L d\big((t,t^2), (t,0)\big)=Lt^2,
\]
a contradiction.
\end{example}

We conclude
with an example illustrating some of the properties listed above (see Example 4.2 in \cite{Miesch2015}).

\begin{example}\label{example_2}
Consider the unit sphere $S^1$ equipped with the induced metric $d$ and the inner metric $d'$.
They are bi-Lipschitz homeomorphic to each other, they are not simply-connected and
$(S^1, d)$ is not geodesic.
According to Lemma 4.3 in \cite{Miesch2015},
an absolute $1$-Lipschitz uniform neighborhood retract
is geodesic and simply-connected, hence in this case
 both spaces do not have this property.

Embedding $(S^1, d)$ into $\R^2$ isometrically we see that
both spaces are not ALRs,
but
the radial projection from a neighborhood of $(S^1, d)$ is Lipschitz,
thus showing that both spaces are ALUNRs and so have property $L$.

Every point in $(S^1,d')$ has a
neighborhood isometric to $(-\tfrac{\pi}{2},\tfrac{\pi}{2})$, which is an absolute $1$-Lipschitz retract, but the whole space is not
even an ALR.
\end{example}

\newpage
\section{$\N$-Approximation}\label{section-N-approx.}

We begin this section with the decomposition lemma already mentioned and used in the introduction 
to prove Proposition \ref{prop_M_approximation} ($\M$-Approximation).

\begin{lemma}\label{lemma-decomposition}
Let $n\geq 1$, and let $X$ be locally compact metric space with local property $L$.
Every $T\in\calI_{n,\c}(X)$ admits a decomposition $T=T_1+\cdots +T_k$ with $T_i\in\calI_{n,\c}(X)$
such that each $\spt(T_i)$ is contained in $\spt(T)$ and has a neighborhood with property $L$.
Suppose in addition that $T\in\I{n}(X)$, then each $T_i\in\I{n}(X)$ as well.
\end{lemma}

\begin{proof}
Suppose that $T\in\I{n}(X)$;
the argument for $T\in\calI_{n,\c}(X)$ is simpler but the one presented here applies as well.
By assumption there exist finitely many points $x_1\ldotss x_k\in \spt(T)$
and radii $r_1\ldotss r_k>0$ such that 
${\spt(T)\subset\bigcup_{i=1}^k\B{x_i}{\tfrac{r_i}{2}}}$
and
each $\B{x_i}{r_i}$ has a neighborhood with {property $L$}.

Take $s_1\in(\tfrac{r_1}{2},r_1)$ such that $T_1\coloneqq T\rr{\B{x_1}{s_1}}\in\I{n}(X)$,
then ${\spt(T_1)\subset \spt(T)}$, $T-T_1= T \rr{(X\setminus \B{x_1}{s_1})}\in\I{n}(X)$
has support in $\spt(T)$ and covered by $\bigcup_{i=2}^k\B{x_i}{\tfrac{r_i}{2}}$.
Then proceed analogously for $r_2\ldotss r_k$.
\end{proof}

We now prove a version of the $\N$-Approximation Theorem
 for an integral 
current whose boundary is already a Lipschitz chain.
We will then upgrade it to Theorem \ref{theorem-N-approx} ($\N$-Approximation)
by first showing that property $L$ is enough to fill arbitrarily small cycles in $X$ (in an appropriate sense).

\begin{proposition}\label{prop-N-approx-dL}
Let $n\geq 1$, let $X$ be a locally compact metric space, and let $U\subset X$ be an open subset with property $L$.
Let $T\in\I{n}(X)$ with $\d T\in \L{n-1}(X)$ and $\spt(T)\subset U$. 
Then for every $\e>0$ there is $R\in \L{n}(X)$ with $\M(T-R)<\e$,
$\d T=\d R$ and $\spt(R)\subset\spt(T)_\e$, in particular $\N(T-R)<\e$.
\end{proposition}

Since $0$-dimensional cycles in $X$ are by definition 
Lipschitz chains, that is, $\Z{0}(X)\subset \L{0}(X)$, it follows that
any $T\in \I{1}(X)$ automatically satisfies the assumptions of this proposition.

\begin{proof}
(See also Figure \ref{figure} for a schematic diagram of the proof.)

Let $K$ denote the closed $\tfrac{\e}{2}$-neighborhood of $\spt(T)$ in $X$,
 without loss of generality we might assume that 
 $K$ is compact and that
 $\spt(T)_\e\subset U$.
Let $\iota\colon K\to \l$ be an isometric embedding
with compact image $K'\coloneqq \iota(K)$.
By property $L$ there exist $\e_0>0$, $L\geq 1$ and an $L$-Lipschitz extension 
\[
g\colon K'_{\e_0}\to X
\]
of $\iota\inv=g\r{K'}$ to the open $\e_0$-neighborhood of $K'$ in $\l$.
According to Proposition \ref{prop_M_approximation} ($\M$-Approximation) we find
$P\in\L{n}(X)$ with 
$\spt(P)\subset \spt(T)_{\e/2}\subset K$ and $\M(T-P)<\tfrac{\e}{6L^n}<\tfrac{\e}{2}$. 

By the metric approximation property of $\l$, Proposition \ref{prop_map_0}, there is a 
finite dimensional subspace $V\subset\l$ and a $1$-Lipschitz 
projection $\pi\colon \l\to V$, such that $\abs{x-\pi(x)}\leq \tfrac{\delta}{2}$ for all $x\in K'$, where
\[
\delta\coloneqq \min\Brace{\frac{\e_0}{2}, \frac{\e}{3nL^n\M(\d T-\d P)},\frac{\e}{4L}},
\]
in particular, $K''\coloneqq \pi(K')\subset K'_\delta\subset K'_{\e_0/2}\cap K'_{\e/(4L)}$.

Now, consider 
\begin{align*}
T'&\coloneqq \push\iota T\in\I{n}(\l), & P'&\coloneqq \push\iota P\in\L{n}(\l),\\
T''&\coloneqq \push\pi T'\in\I{n}(V), & P''&\coloneqq \push\pi P'\in\L{n}(V).
\end{align*}
Note that $\d T'\in\L{n-1}(\l)$, $\d T''\in\L{n-1}(V)$, $\M(T''-P'')\leq \M(T'-P')=\M(T-P)$, the supports of $T',P'$ are contained in $K'$ and the supports of $T'', P''$ are contained in $K''$.

Let $H\colon \i\times K'\to \l$ denote the affine homotopy between $\id_{K'}$ and $\pi\r{K'}$, and let
$W\coloneqq \push H\big(\ii\times (\d T'-\d P')\big)\in \L{n}(\l)$ be the affine filling of 
$\big(\d T''-\d P''\big)-\big(\d T'-\d P'\big)$
as defined in Section \ref{section-homotopies}.
Note that 
$H(t,\cdot)\colon K'\to\l$
is $1$-Lipschitz for all $t\in\i$ and
$H(\cdot, x)\colon \i\to\l$ 
has length at most $\delta/2$ for all $x\in K'$.
Therefore the support $\spt(W)$ of $W$ is contained in 
$K'_\delta\subset K'_{\e_0/2}\cap K'_{\e/(4L)}$ and its mass is bounded by
\[
\M(W)\leq n\frac{\delta}{2}\M(\d T'-\d P')\leq \frac{\e}{6 L^n}.
\]
As $\d(T''-P'')=\d T''-\d P'' \in\L{n-1}(V)$, by Lemma \ref{lemma_FF_Lemma5.7_0} we find $S\in\I{n+1}(V)$ satisfying
\begin{align*}
&\N(S)\leq \eta\coloneqq \min\Brace{\frac{\e_0}{2}, \frac{\e}{6 L^n}}<\frac{\e}{4L},\\
&\spt(S)\subset \spt(T''-P'')_\eta\subset K''_\eta
	\subset K'_{\eta+\delta}\subset K'_{\e_0}\cap K'_{\e/(2L)},\\
&T''-P''-\d S\in\L{n}(V),
\end{align*}
(in fact, $\spt(S)$ is contained in the open $\eta$-neighborhood of $\spt(T''-P'')$ in $V$).

Finally, note that $T''-P''-\d S$ and $W$ are both Lipschitz $n$-chains with supports in 
$K'_{\e_0}\cap K'_{\e/(2L)}$, 
so that 
$\push g \big(T''-P''-\d S -W\big)\in\L{n}(X)$ 
is well defined,
has support in
$g(K'_{\e/(2L)})\subset K_{\e/2}$,
mass
\begin{align*}
\M\big(\push g \big(T''-P''-\d S -W\big)\big)
&\leq L^n\Big(\M(T''-P'')+\M(\d S)+ \M(W)\Big)\leq \frac{\e}{2}
\end{align*}
and boundary
\begin{align*}
\d \big(\push g \big(T''-P''-\d S -W\big)\big)
&=\push g \big(\d T''-\d P''-\big(\d T''-\d P''- \d T'+\d P'\big)\big)\\
&=\push g \big(\d T'-\d P'\big)\\
&=\d T-\d P,
\end{align*}
where in the last equality we have used that $g\r{K'}=\iota\inv$.
Overall, the Lipschitz $n$-chain
\[
R\coloneqq P+\push g \big(T''-P''-\d S-W\big)\in \L{n}(X)
\] 
satisfies
$\M(T-R)\leq \M(T-P)+\M(R-P)< \tfrac{\e}{2}+\tfrac{\e}{2}=\e$,
$\d R=\d P+\d T-\d P=\d T$ and
$\spt(R)\subset \spt(P)\cup K_{\e/2}\subset \spt(T)_\e$.
\end{proof}

Notably, this proposition implies 
a version of Theorem \ref{theorem-N-approx} ($\N$-Approximation) for cycles.

\begin{corollary}\label{cor-N-approx-dZ}
Let $n\geq 1$, let $X$ be a locally compact metric space, and let $U\subset X$ be an open subset with property $L$.
Then for every $Z\in\Z{n}(X)$ with $\spt(Z)\subset U$,
and every $\e>0$, there is $R\in \L{n}(X)$ with 
$\d R=0$, 
$\N(Z-R)<\e$,
and $\spt(R)\subset\spt(Z)_\e$.
\end{corollary}

If we assume that $X$ has local property $L$ we are not able to prove statements like 
Proposition \ref{prop-N-approx-dL} and Corollary \ref{cor-N-approx-dZ} for integral currents
in $X$ whose boundary is a Lipschitz chain but without restrictions on their supports.
This is because the decomposition $T=T_1+\cdots +T_k$ of Lemma \ref{lemma-decomposition} does
not preserve some properties of the boundary.
More specifically, even if $\d T\in \L{n-1}(X)$ or $\d T=0$,
it might happen that $\d T_i\not \in \L{n-1}(X)$ or $\d T_i\neq 0$ for some $i$.

Nonetheless, if $n=1$ then it does holds that $\d T_i\in \Z{0}(X)\subset \L{0}(X)$ for all $i$
and hence we can apply Proposition \ref{prop-N-approx-dL} to each component to obtain $R_1\ldotss R_k\in \L{1}(X)$
with $\M(T_i-R_i)<\e/k$, $\d R_i=\d T_i$, and $\spt(R_i)\subset \spt(T_i)_\e$ for all $i$.
This implies a stronger version of Theorem \ref{theorem-N-approx} for $n=1$.  

\begin{corollary}\label{cor-N-approx-n=1}
Let $X$ be a locally compact metric space with local property $L$. 
Then for every $T\in\I{1}(X)$ and $\e>0$ there exists $R\in\L{1}(X)$
with $\d P=\d T$, $\N(T-P)<\e$ and $\spt(P)\subset \spt(T)_\e$.

In particular, $\d P=0$ whenever $T\in\Z{1}(X)$.
\end{corollary}

As in the proof of Proposition \ref{prop-N-approx-dL}, given an open subset $U$ of $X$ with property $L$ and 
a compact subset $K\subset U$, we can consider the isometric embedding
$\iota\colon K\to \l$ and the Lipschitz extension 
$g\colon K'_{\e_0}\to X$
of $\iota\inv=g\r{K'}$. 
If $Z\in\Z{n}(X)$ has support in $K$ we can fill $\push{\iota}Z$ 
in $\l$ and if $\M(Z)$ is small enough we expect to be able to 
push the filling back into $X$ using $g$.
The next proposition establishes this result precisely.

\begin{proposition}\label{prop-small-fillings}
Let $n\geq 1$, let $X$ be a locally compact metric space, and let $U\subset X$ be an open subset with property $L$.
Then for every compact subset $K\subset U$ and $\e>0$ there exists $M>0$
such that every $Z\in \Z{n}(X)$ with 
$\spt(Z)\subset K$ and $\M(Z)<M$ possesses a filling $S\in \I{n+1}(X)$
with $\spt(S)\subset \spt(Z)_\e$ and $\M(S)<\e$.
\end{proposition}

The proof is similar to that of Proposition \ref{prop-N-approx-dL} in the sense
that we are going to exploit known properties of a finite dimensional subspace $V$ of $\l$,
namely that $V$ admits a Euclidean isoperimetric inequality for $\Z{n}(V)$,
and the existence of solutions to the Plateau problem.
Therefore every $Z\in\Z{n}(V)$ admits a filling $S\in\I{n+1}(V)$ 
with $\M(S)\leq C\M(Z)^{(n+1)/n}$ and support $\spt(S)$
within distance at most $(n+1)C\M(Z)^{1/n}$ from $\spt(Z)$, where $C$
is a constant depending only on $n$.
This result was shown for classical integral currents in \cite{Federer1960}
and holds more generally for metric currents in the sense of Ambrosio-Kirchheim \cite{Ambrosio2000} and Lang \cite{Lang2011}
(see Theorem~1.2 and Theorem~1.6 in \cite{Wenger2005},
as well as Section~2.7 and Section~2.8 in \cite{KleinerLang2018}).

\begin{proof}
Let $\iota\colon K\to \l$ be an isometric embedding with compact image $K'\coloneqq\iota(K)$.
By property $L$ there exist $\e_0>0$, $L\geq 1$ and an $L$-Lipschitz extension
$g\colon K'_{\e_0}\to X$ of $\iota\inv=g\r{K'}$.
We might assume that $\e\leq 1 $ and $\e_0<\e/L$ so that $g(K_{\e_0}')\subset K_\e$.
By the metric approximation property of $\l$, there exist a finite dimensional subspace $V\subset\l$ and a $1$-Lipschitz map
$\pi\colon \l\to V$ such that $\abs{x-\pi(x)}\leq \e_{0}/4$ for all $x\in K'$,
in particular $K''\coloneqq \pi(K')\subset K'_{\e_0/2}$.

Let $C\geq 1$ be the constant from above,
set
\[
M\coloneqq \min\Brace{\Big(\frac{\e_0}{2(n+1)C}\Big)^n,\frac{\e}{(C+\frac{n+1}{4}\e_0)L^{n+1}}}
\]
and let $Z\in\Z{n}(X)$ with $\spt(Z)\subset K$
and $\M(Z)<M$.

Consider 
\begin{align*}
Z'\coloneqq \push{\iota}Z\in \Z{n}(\l),
&& 
Z''\coloneqq \push{\pi}Z'\in\Z{n}(V),
\end{align*}
which have supports in $K'$ and $K''$, respectively, and satisfy $\M(Z'')\leq \M(Z')=\M(Z)<M$.
Let $H\colon \i\times \spt(Z')'\to \l$ denote the affine homotopy 
between $\id_{\spt(Z')}$ and $\pi\r{\spt(Z')}$, and let
$Q\coloneqq \push{H}(\ii\times Z')\in\I{n+1}(\l)$
be the affine filling of $Z''-Z'$, as defined in Section \ref{section-homotopies}.
Note that $H(t,\cdot)\colon \spt(Z')\to\l$ is $1$-Lipschitz 
for all $t\in \i$
and
$H(\cdot, x)\colon \i\to \l$ has length at most $\e_0/4$ for all $x\in \spt(Z')$.
Thus the support $\spt(Q)$ of $Q$ is contained in $\spt(Z')_{\e_0}\subset K'_{\e_0}$ 
and its mass is bounded by
\[
\M(Q)\leq (n+1)\frac{\e_0}{4}\M(Z')<\tfrac{n+1}{4}\e_0M.
\]
As noted above $Z''$ possesses a filling $S''\in\I{n+1}(V)$
with mass
\[
\M(S'')\leq C\M(Z'')^{\tfrac{n+1}{n}}<CM^{\tfrac{n+1}{n}}\leq CM
\]
and support within distance at most $(n+1)C\M(Z'')^{1/n}<\e_0/2$ from $\spt(Z'')\subset \pi(\spt(Z'))$,
in particular it is contained in 
$\pi(\spt(Z'))_{\e_0/2}\subset \spt(Z')_{\e_0}\subset K_{\e_0}'$.

Finally, $S''$ and $Q$ have support in $\spt(Z')_{\e_0}\subset K_{\e_0}'$ so that 
$S\coloneqq \push{g}(S''-Q)\in \I{n+1}(X)$ is well defined,
has support in $g(\spt(Z')_{\e_0})\subset \spt(Z)_\e$, has boundary
$\d S=\push{g}(\d S''-\d Q)=\push{g}(Z''-Z''+Z')=Z$
and its mass is bounded by
\[
\M\big(S\big)\leq L^{n+1}\big(\M(S'')+\M(Q)\big)
< L^{n+1}\Big(C+\tfrac{n+1}{4}\e_0\Big)M\leq \e.
\qedhere
\]

\end{proof}

We are now in a position to upgrade Proposition \ref{prop-N-approx-dL}
to any current ${T\in\I{n}(X)}$.

\begin{proposition}\label{prop-N-approximation-U}
Let $n\geq 1$, let $X$ be a locally compact metric space, and let $U\subset X$ be an open subset with property $L$.
Then for every $T\in\I{n}(X)$ with $\spt(T)\subset U$,
and every $\e>0$, there is $P\in\L{n}(X)$
with $\N(T-P)<\e$ and $\spt(P)\subset \spt(T)_\e$.
\end{proposition}

A stronger conclusion holds already for the $1$-dimensional case from Proposition \ref{prop-N-approx-dL}
and the comment thereafter,
so that in the proof we can assume that $n\geq 2$ and apply
 Proposition \ref{prop-small-fillings} in dimension $n-1\geq 1$.

\begin{proof}
Suppose $n\geq 2$.
Let $K$ denote the closed $\tfrac{\e}{2}$-neighborhood of $\spt(T)$ in $X$,
 without loss of generality we might assume that 
 $K$ is compact and that
 $\spt(T)_\e\subset U$.
Let $M>0$
be the constant of Proposition \ref{prop-small-fillings}
for $K$ and $\e/4$;
up to decreasing it we might assume that $M\leq\e/4$.

Consider $T'\coloneqq \d T\in\Z{n-1}(X)$
and note that $\spt(T')_{\e/2}\subset\spt(T)_{\e/2}\subset K\subset U$.
By Proposition \ref{prop-N-approx-dL} we can find $P'\in\L{n-1}(X)$ with $\d P'=\d T'(=0)$,
$\M(T'-P')<M\leq\e/4$ and $\spt(P')\subset \spt(T')_{\e/4}\subset K$.

According to Proposition \ref{prop-small-fillings} and the choice of $M$, there exists a filling $S\in\I{n}(X)$
of $T'-P'$ with $\M(S)<\e/4$ and $\spt(S)\subset \spt(T')_{\e/2}\subset U$.

Note that $T-S\in\I{n}(X)$ has support contained in $\spt(T)_{\e/2}\subset U$ and boundary
$\d(T-S)=T'-(T'-P')=P'\in\L{n-1}(X)$ 
so applying Proposition \ref{prop-N-approx-dL}
a second time we find 
$P\in\L{n}(X)$
with $\M(T-S-P)<\e/2$, 
${\d P=\d(T-S)=P'}$, and
$\spt(P)\subset \spt(T-S)_{\e/2}\subset \spt(T)_\e$.

Therefore $P$ satisfies:
\begin{align*}
\M(T-P)&\leq \M(T-S-P)+\M(S)<\frac{\e}{2}+\frac{\e}{4}\leq \e,\\
\M(\d T-\d P)&=\M(T'-P')<M<\e.
\end{align*}
\end{proof}

The proof of Theorem \ref{theorem-N-approx} ($\N$-Approximation) now follows by
combining Lemma \ref{lemma-decomposition} and Proposition \ref{prop-N-approximation-U}.

\begin{proof}[Proof of Theorem \ref{theorem-N-approx}]
Let $X$ be a locally compact metric space with local property $L$, $T\in\I{n}(X)$ and $\e>0$.
By Lemma \ref{lemma-decomposition} we can write $T=T_1+\cdots+T_k$ with each $T_i\in\I{n}(X)$ having support
contained in both $\spt(T)$ and in an open subset of $X$ having property $L$.
By Proposition \ref{prop-N-approximation-U} there exist $P_i\in\L{n}(X)$ with
$\N(T_i-P_i)<\e/k$ and 
$\spt(P_i)\subset \spt(T_i)_\e\subset \spt(T)_\e$, so that
 ${P\coloneqq P_1+\cdots +P_k\in\L{n}(X)}$ is the desired Lipschitz approximation of $T$
 with $N(T-P)<\e$ and $\spt(P)\subset \spt(T)_\e$.
\end{proof}

\begin{figure}[H]
\centering
\includegraphics[scale=1]{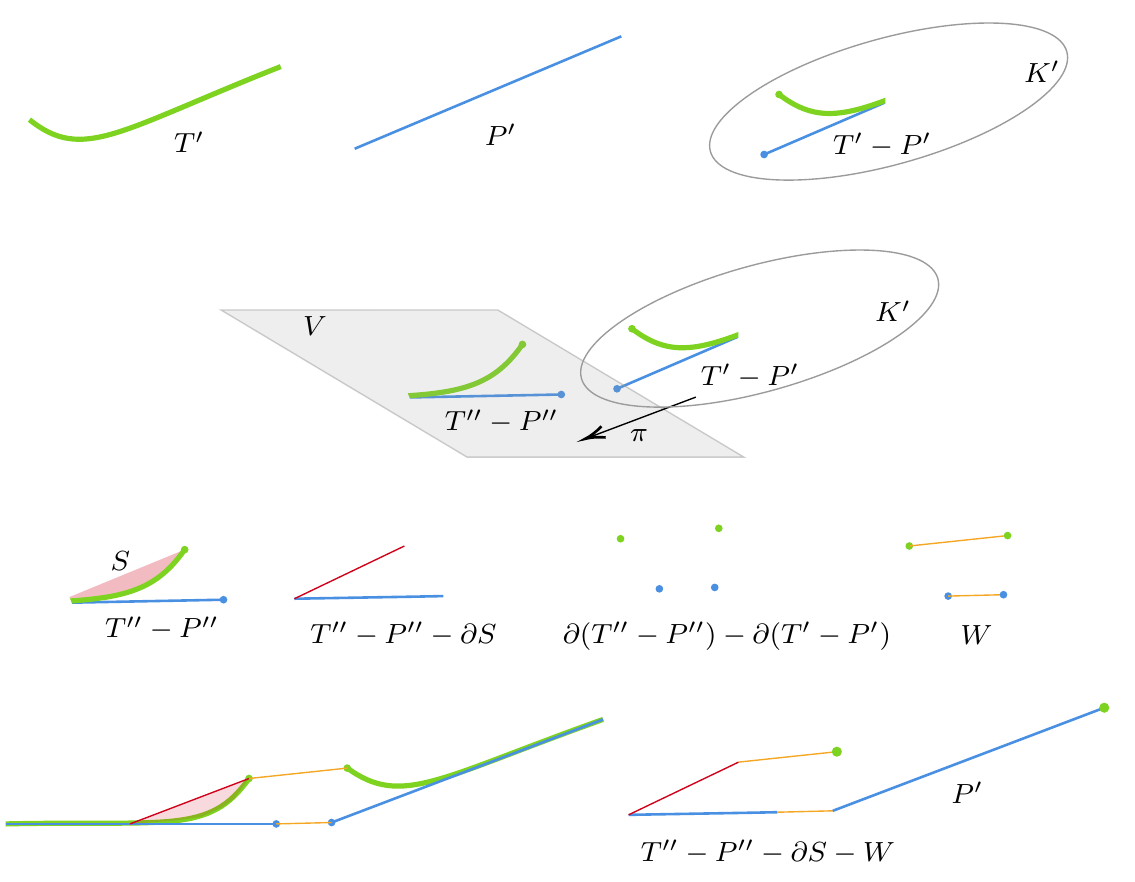}
\caption{In this figure we illustrate schematically some of the currents used in the proof of Proposition \ref{prop-N-approx-dL}.
By construction $T'-P'$ is a current in $K'\subset \l$ and has small mass $M(T'-P')$;
$T''-P''-\d S$ is a current in a finite dimensional space $V$ and $T''-P''-\d S$ is in $\L{n}(V)$.
$W$ is the affine homotopy filling of $\d(T''-P'')-\d(T'-P')$ and
overall 
$P'+(T''-P''-\d S-W)$ is a Lipschitz chain with boundary boundary $\d T'$
 and the error term $\M(T''-P''-\d S-W)$ is small.
}\label{figure}
\end{figure}
\newpage
\printbibliography

\Addresses
\end{document}